\documentclass[reqno]{amsart}

\usepackage{amsmath,amssymb,amsthm}
\usepackage{esint}

\usepackage{tikz}

\usepackage{caption}
\usepackage{graphicx}
\usepackage{graphics}

\newtheorem{theorem}{Theorem}

\newtheorem{lemma}{Lemma}

\newtheorem*{proposition}{Proposition}

\begin{document}

\title[]{Random Growth via Gradient Flow Aggregation}

\author[]{Stefan Steinerberger}
\address{Department of Mathematics, University of Washington, Seattle, WA 98195, USA} \email{steinerb@uw.edu}


\begin{abstract} We introduce Gradient Flow Aggregation (GFA), a random growth model. Given a set of existing particles $\left\{x_1, \dots, x_n\right\} \subset \mathbb{R}^2$, a new particle arrives from a random direction at $\infty$ and flows in direction $\nabla E$ where 
$$ E(x) = \sum_{i=1}^{n} \frac{1}{\|x-x_i\|^{\alpha}} \qquad \mbox{where} ~0 < \alpha < \infty.$$
The case $\alpha = 0$ will refer to the logarithmic energy $- \sum\log \|x-x_i\|$. Particles stop once they are at distance 1 of one of the existing particles at which point they are added to the set and remain fixed for all time. We prove, under a non-degeneracy assumption, a Beurling-type estimate which, via Kesten's method, can be used to deduce sub-ballistic growth for $0 \leq \alpha < 1$
$$\mbox{diam}(\left\{x_1, \dots, x_n\right\}) \leq c_{\alpha} \cdot n^{\frac{3 \alpha +1}{2\alpha + 2}}.$$
 This is optimal when $\alpha =0$. The case $\alpha = 0$ leads to a `round' full-dimensional tree. The larger the value of $\alpha$ the sparser the tree. Some instances of the higher-dimensional setting are also discussed.\end{abstract}

\maketitle

\section{Introduction and Results}

\subsection{Aggregation models}
The purpose of this paper is to introduce a model describing random growth of particles aggregating in the plane. We will almost exclusively work in $\mathbb{R}^2$ and we will consider particles to be disks of radius $1/2$ (which are then often identified with their center).

\begin{center}
\begin{figure}[h!]
\begin{tikzpicture}
\node at (0,0) {\includegraphics[width=0.32\textwidth]{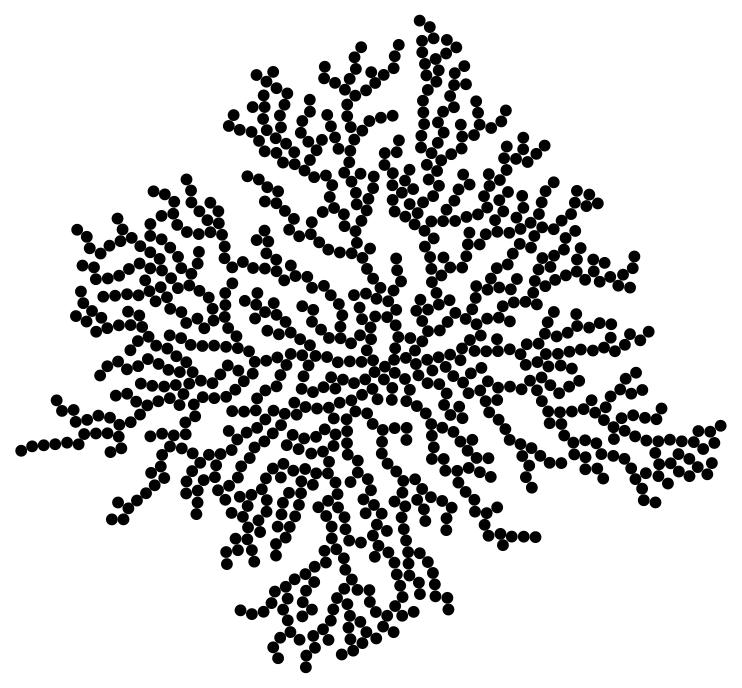}};
\node at (0, -2.2) {$\alpha = 0$};
\node at (4.3,0) {\includegraphics[width=0.3\textwidth]{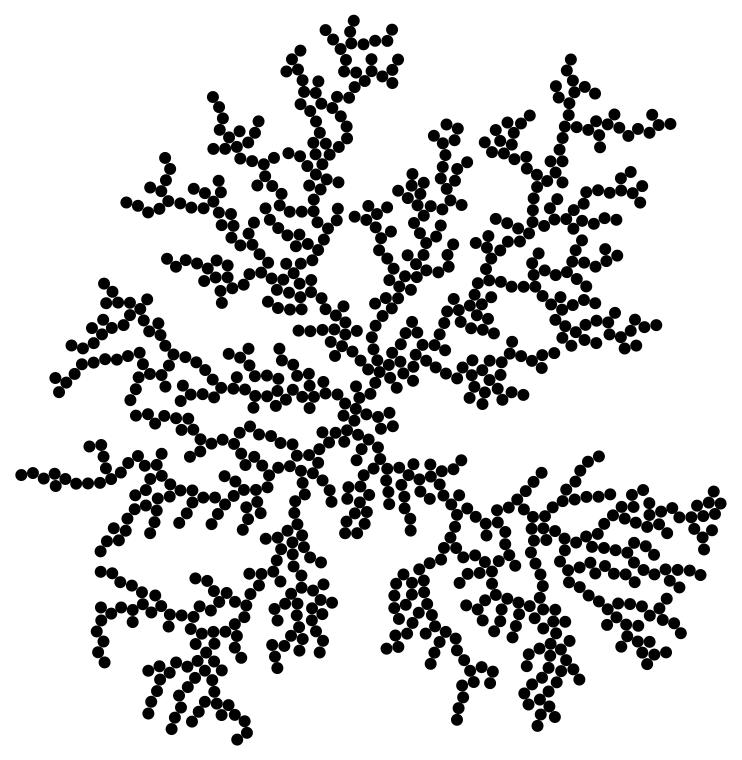}};
\node at (4.3, -2.2) {$\alpha = 1$};
\node at (8.5,0) {\includegraphics[width=0.29\textwidth]{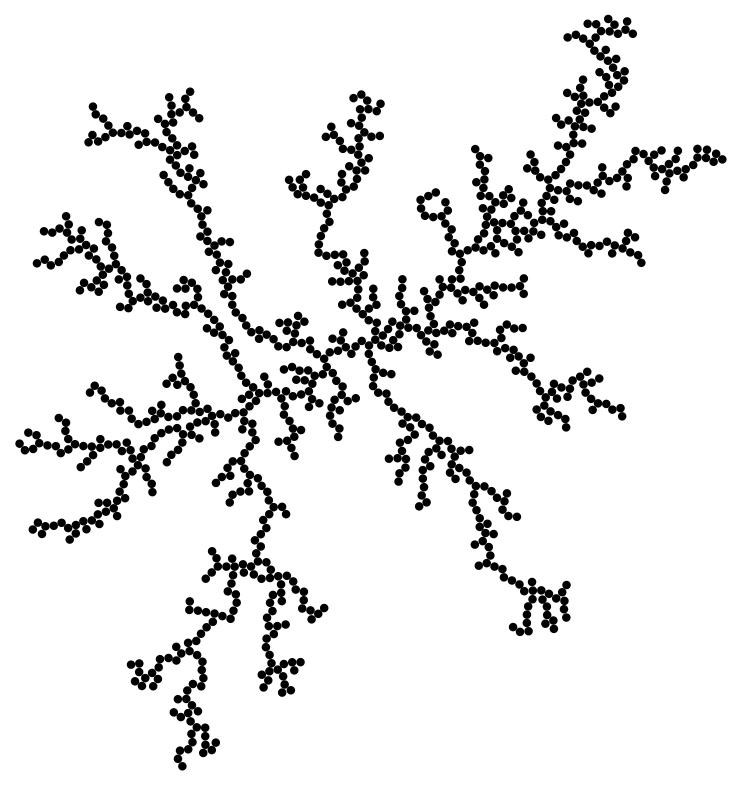}};
\node at (8.5, -2.2) {$\alpha = 2$};
\end{tikzpicture}
\caption{Simulation of Gradient Flow Aggregation (GFA) for $n=1000$ particles with various parameters of $\alpha$.}
\label{fig:1}
\end{figure}
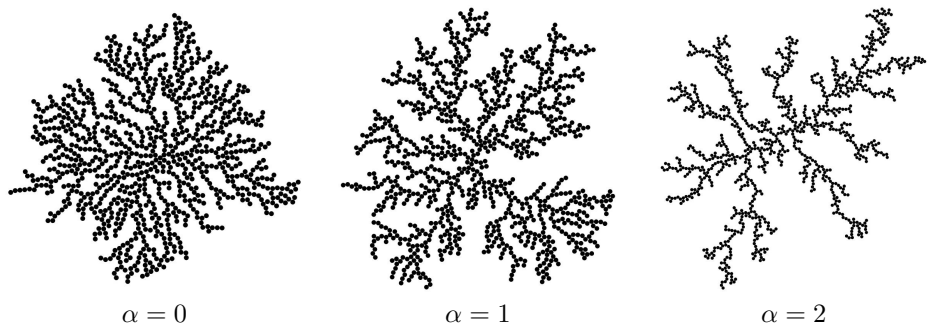
\end{center}

Perhaps the most celebrated random growth model is Diffusion Limited Aggregation (DLA) introduced by Witten \& Sanders \cite{witten} in 1981. DLA models random growth by assuming that a particle coming from `infinity' performs a random walk until it first touches an already existing set of particles at which point it gets stuck for all time. The model is frequently considered on the lattice $\mathbb{Z}^d$ where random walks are particularly easy to define. The main results are due to Kesten \cite{kesten0, kesten1, kesten2} who proved that the diameter of a cluster of $n$ particles under DLA on the lattice $\mathbb{Z}^2$ is $\lesssim n^{2/3}$ (along with generalizations to $\mathbb{Z}^d$). No non-trivial lower bound (better than $\gtrsim n^{1/2}$) is known.
Many other models have been proposed, we specifically mention the much Eden model \cite{eden},  the Vold--Sutherland model \cite{suh, vold}, the Dielectric Breakdown Model (DBM) \cite{losev, nie} and the Hastings-Levitov model \cite{hastings}.
They have received a lot of attention because of the intricate emerging complexity and the simplicity of the setup: nonetheless, there are relatively few rigorous results. In particular, DLA remains
 `notoriously resilient to rigourous analysis' (Benjamini \& Yadin \cite{ben}). We were motivated by the question of whether the underlying probability theory (a random Brownian walker coming from infinity) could be replaced by more elementary ingredients and whether this would lead to a model where the corresponding theory would also simplify (and thus, hopefully, allow for further insight). This motivated GFA which seems to exhibit emergent complexity similar to that of many of the previous models. Additionally, it admits an analogous theory (a Beurling-style estimate leading to a growth bound via Kesten's method) where all the constituent elements are `elementary'.

\subsection{Gradient Flow Aggregation}
We start with a single particle $x_1 \in \mathbb{R}^2$. If we already have $\left\{x_1, \dots, x_n \right\} \subset \mathbb{R}^2$, then a new particle $x_{n+1}$ is created randomly at `infinity' (uniformly over all directions, made precise below) and flows in direction of $\nabla E$ (gradient \textit{ascent}) where
$$ E(x) =  \sum_{i=1}^{n}\frac{1}{\|x - x_i\|^{\alpha}} \qquad \mbox{where}~0 < \alpha < \infty.$$

\begin{center}
\begin{figure}[h!]
\begin{tikzpicture}
\node at (0,0) {\includegraphics[width=0.35\textwidth]{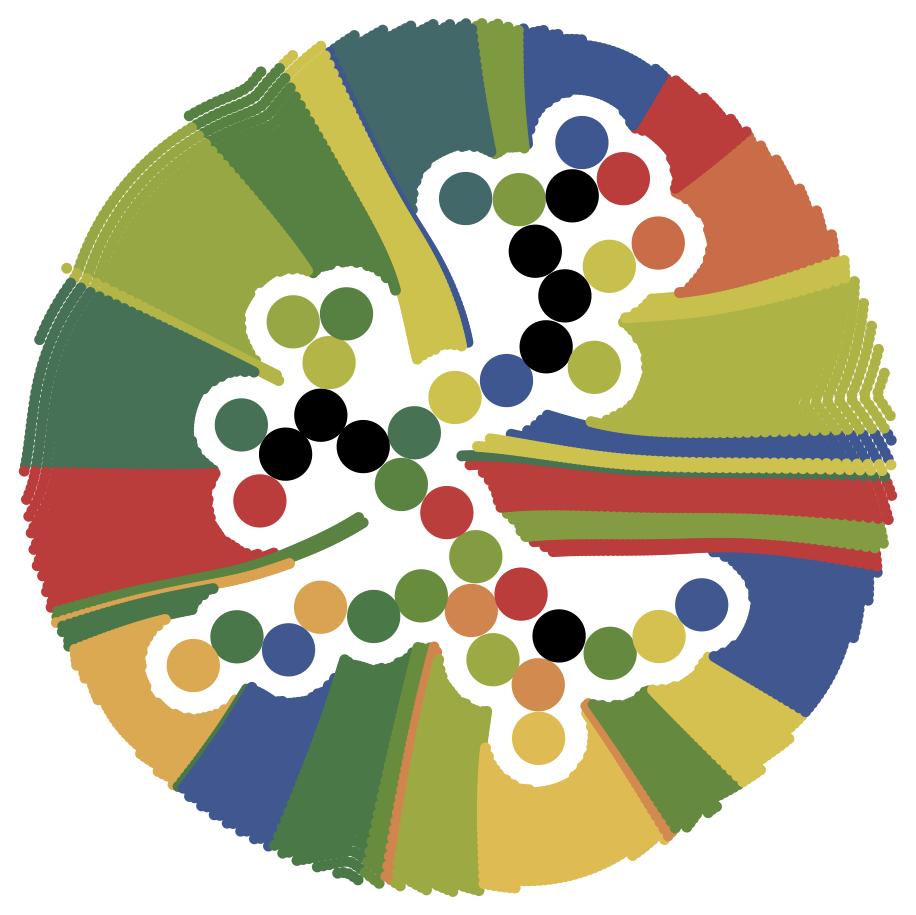}};
\node at (0, -2.5) {$\alpha = 0$};
\node at (6,0) {\includegraphics[width=0.35\textwidth]{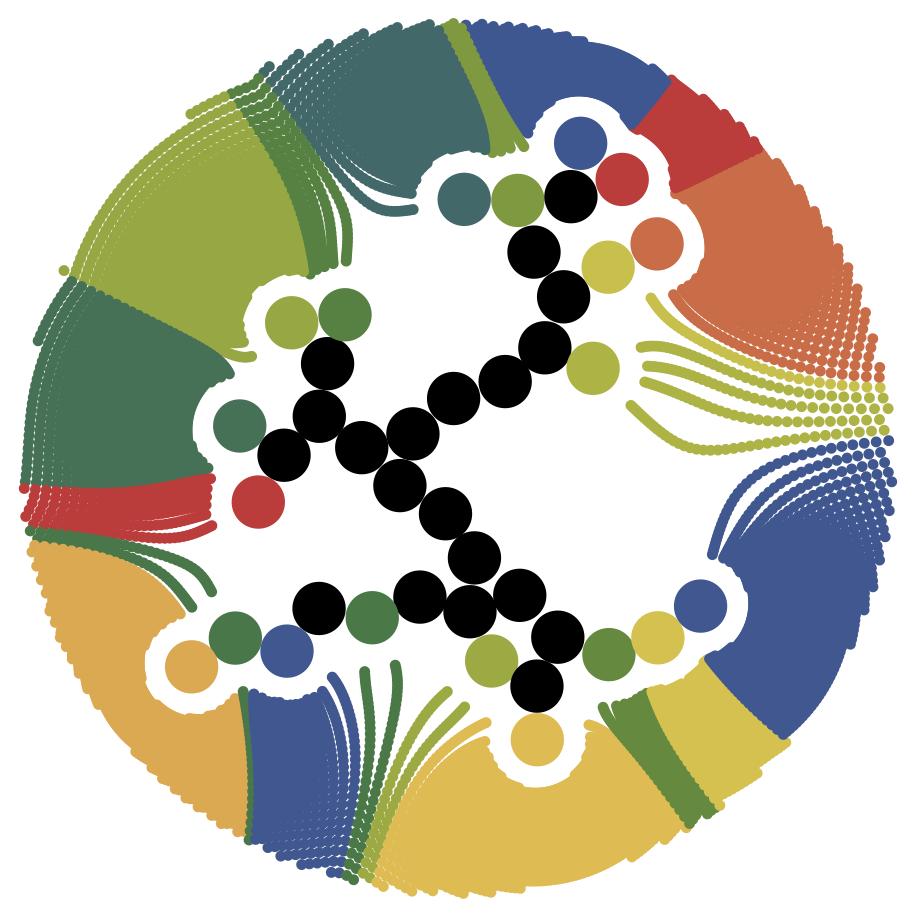}};
\node at (6, -2.5) {$\alpha = 4$};
\end{tikzpicture}
\caption{A fixed cluster of 40 particles. Gradient descent lines are colored by the color of the particle they eventually attach to (particles colored black were not hit by any gradient descent in the simulation). Likelihoods are determined by gradient flow lines at infinity. As $\alpha$ increases, the exposed endpoints gain more mass.}
\label{fig:2}
\end{figure}
\end{center}

The particle flows until it reaches distance 1 to one of the existing points for the first time; at that point, the flow stops and the particle then remains at that location for all time. There is a natural interpretation of the model where new particles are drawn to the existing particles via a `field' generated by the existing particles. The only randomness in the model is the starting angle at infinity of the new particle (a single uniformly distributed random variable on $[0,2\pi]$). In particular, the model requires less `random input' (a single random variable) than DLA (where an entire Brownian path is required). There are two obvious modifications of the energy $E$ at the endpoints $\alpha \in \left\{0, \infty\right\}$. We define
\begin{align*}
 E(x) &= \sum_{i=1}^{n} \log\left( \frac{1}{\|x-x_i\|}\right) \qquad  \mbox{when}~\alpha = 0 \\
 E(x) &= \max_{1 \leq i \leq n} \frac{1}{\|x-x_i\|} \quad \qquad  \quad \mbox{when} ~\alpha =\infty.
 \end{align*}
 It remains to consider how to proceed with critical points $\nabla E(x^*) = 0$ which may have the effect of trapping a gradient flow. In practice, this is not actually an issue. We work, for the remainder of the paper, under a non-degeneracy assumption.
 \begin{quote}
\textbf{Definition.} We say $\left\{x_1, \dots, x_n\right\} \subset \mathbb{R}^2$ satisfies property $(P)$ if $E(x) = \sum_{i=1}^{n} \|x-x_i\|^{-\alpha}$ has a finite number of critical points.
\end{quote}
A result of Gauss implies that when $\alpha = 0$, then every set has property $(P)$. It is widely assumed to be true in general (a version of this question is known as `Maxwell's problem'). Since it stands to reason that it is either true for all points or, at very least, typically true (in the sense of a hypothetical exceptional set being rare), it is a rather weak assumption; we comment more on this in \S 2.1.

\begin{center}
\begin{figure}[h!]
\begin{tikzpicture}
\node at (0,0) {\includegraphics[width=0.6\textwidth]{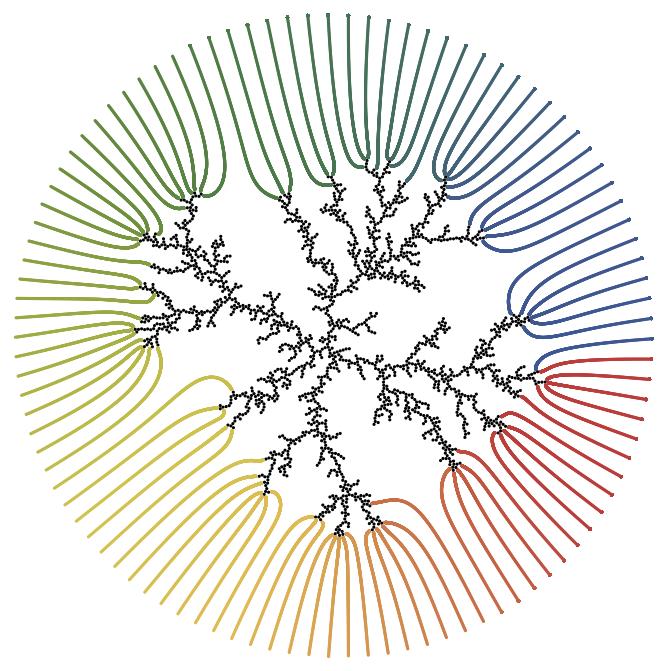}};
\end{tikzpicture}
\caption{A tree grown with $\alpha =2$ and the trajectory of 100 incoming gradient flows (equispaced in angle). Much like in other growth models, there is a tendency to avoid existing valleys.}
\label{fig:3}
\end{figure}
\end{center}

 It remains to give a precise definition of what it means for a new particle to appear `randomly at infinity'. For any $r \gg 1$, we can consider points at distance $\|x\| = r$ from the origin (with uniform probability)
and ask for the likelihood $p_i(r)$ that gradient flow started in the random point gets ends up being attached to $x_i$ (meaning that the gradient ascent flow reaches distance 1 to $x_i$ and this is the first time the gradient flow is distance 1 to any of the points in the existing set). The limit as $r \rightarrow \infty$ exists, the simplest statement one can make in this direction is as follows.

\begin{theorem}
Assuming property (P), the limits $\lim_{r \rightarrow \infty} p_i(r)$ exist and sum to 1.
\end{theorem}

We quickly comment on the role of $\alpha$ (illustrated in Fig. \ref{fig:1}, Fig. \ref{fig:2} and Fig. \ref{fig:4}). Larger values of $\alpha$ lead to a stronger pull by nearby particles. Philosophically, it should be similar to the parameter $\eta$ in the dielectric breakdown model (DBM).  We may think of the newly incoming particle as being pulled by the existing particles with a force being determined by the distance: particles that are closer exert a stronger pull. Larger values of $\alpha$ lead to sparser trees.

\begin{center}
\begin{figure}[h!]
\begin{tikzpicture}
\node at (0,0) {\includegraphics[width=0.5\textwidth]{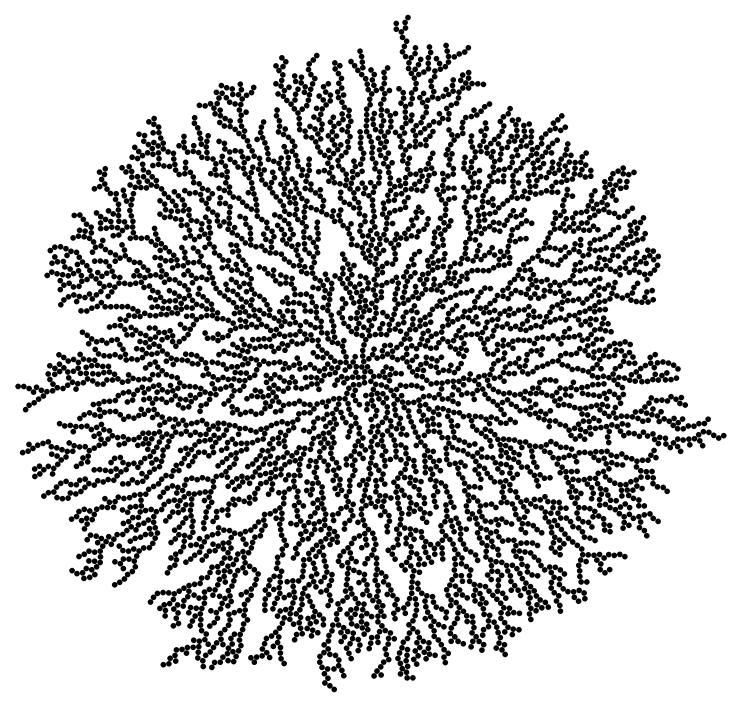}};
\node at (0, -3.3) {$\alpha = 0$};
\node at (6.3,0) {\includegraphics[width=0.5\textwidth]{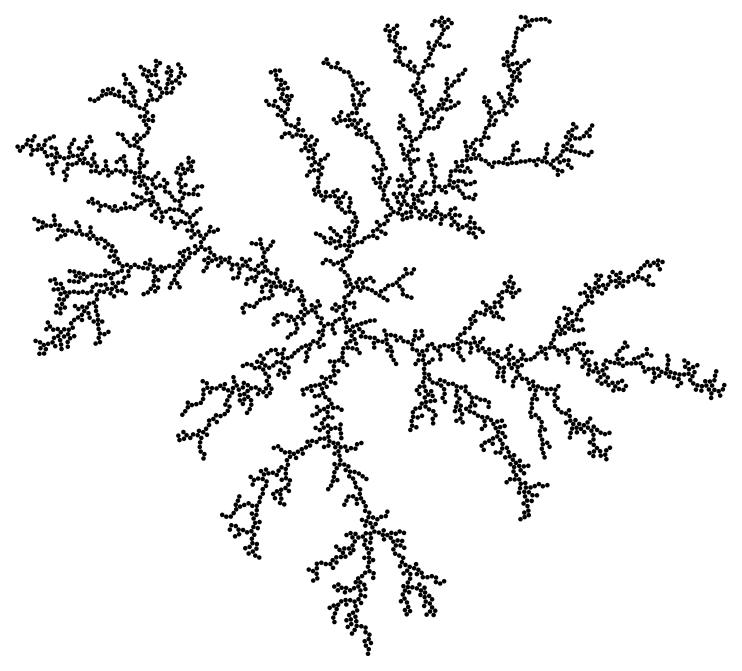}};
\node at (6.3, -3.3) {$\alpha = 2$};
\end{tikzpicture}
\caption{GFA for $n=5000$ particles with $\alpha=0$ (left) and $n=2500$ particles with $\alpha =2$. Theorem 3 implies that the structure on the left grows as slowly as possible, $\mbox{diam}(x_1, \dots, x_n) \leq c \sqrt{n}$.}
\label{fig:4}
\end{figure}
\end{center}

\vspace{-20pt}

\subsection{The Beurling estimate.} The crucial estimate in the theory of DLA is Beurling's estimate for harmonic measure which provides a uniform upper bound for the likelihood of a new incoming particle attaching itself to any fixed particle. We prove an analogous result in the setting of GFA for $0 \leq \alpha \leq 1$.

\begin{theorem}[Beurling estimate] Let $0 \leq \alpha \leq 1$ and suppose $\left\{x_1, \dots, x_n\right\} \subset \mathbb{R}^2$ satisfies property (P). Then, for some $c_{\alpha}>0$ depending only on $\alpha$,
$$ \max_{1 \leq i \leq n} \mathbb{P}\left(\emph{new particles hits}~x_i\right) \leq c_{\alpha}  \cdot n^{\frac{\alpha-1}{2\alpha+2}}.$$
\end{theorem}

The result is optimal for $\alpha = 0$ where it gives rate $\lesssim n^{-1/2}$.  The estimate gives nontrivial results all the way up to $\alpha =1$. There is no reason to assume the result is optimal, we comment on this more extensively after the proof.

\subsection{The growth estimate.} Beurling-type estimates can be used to derive growth bounds on the diameter. Originally due to Kesten \cite{kesten0, kesten1, kesten2} (`Kesten's method'), this is now well understood: we adapt a robust and very general framework developed by Benjamini-Yadin for DLA on graphs \cite{ben}.
\begin{theorem}[Growth Bound] Let $0 \leq \alpha \leq 1$, let $(x_k)_{k=1}^{\infty}$ be obtained via GFA and satisfy property (P) throughout. Then, with high probability,
$$ \emph{diam} \left\{x_1, \dots, x_n \right\} \leq c_{\alpha} \cdot n^{\frac{3 \alpha +1}{2\alpha + 2}}.$$
\end{theorem}
This is optimal for $\alpha = 0$: any set of $n$ points in $\mathbb{R}^2$ satisfying $\|x_i - x_j\| \geq 1$ has diameter at least $c \cdot n^{1/2}$. Thus, GFA with $\alpha = 0$ grows at the slowest possible rate (see Fig. \ref{fig:4}) and GFA with $0 < \alpha \ll 1$ small behaves similarly. Kesten's method allows to translate Beurling estimates into growth estimates but it's clear that the argument is inherently lossy: this is the case for DLA and it is equally the case for GFA. It would be very desirable if the simplified framework of GFA were to simplify the investigation of new arguments that might reasonably break this barrier.

\subsection{The case $\alpha \rightarrow \infty$} One nice aspect of GFA is that the case of $\alpha = \infty$ is particularly easy to describe and simulate (see Fig. \ref{fig:5}). It should be related, in spirit, to the behavior of DBM for a suitable $\eta$ (maybe, see \cite{phase}, with a value of $\eta$ close to $4$?).
We first describe the model. 
There is a natural limit in that
$$  \lim_{\alpha \rightarrow \infty} \left( \sum_{i=1}^{n}\frac{1}{\|x - x_i\|^{\alpha}} \right)^{\frac{1}{\alpha}}  = \max_{1 \leq i \leq n} \frac{1}{\|x-x_i\|^{}}.$$
This suggests a particularly natural limiting process: given $\left\{x_1, \dots, x_n\right\} \subset \mathbb{R}^2$, create the next point $x_{n+1}$ by adding a `random point at infinity' (in the same way as above) and then moving it along $\nabla E$. In this case, the gradient $\nabla E$ (which is defined almost everywhere) is simply going to point to the nearest particle in the set. It is not difficult to see that this never changes along the entire flow. This means the particle follows a straight line until the corresponding disks touch and the process stops.
The geometry in this case becomes a lot simpler to analyze, in particular new particles only attach themselves to the convex hull of the existing particles (and the likelihoods can be computed in terms of opening angles).

\begin{center}
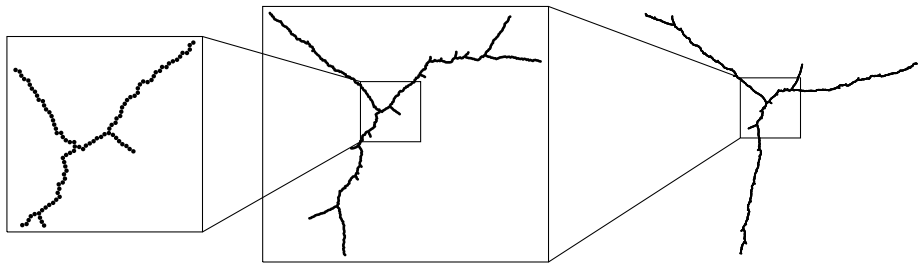
\begin{figure}[h!]
\begin{tikzpicture}
\node at (0,0) {\includegraphics[width= 0.2\textwidth]{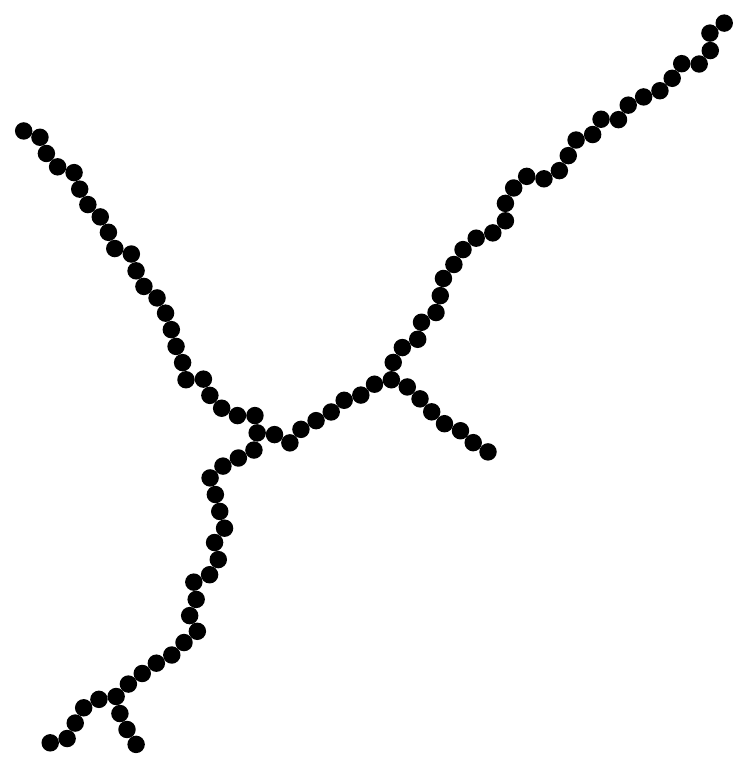}};
\node at (4,0) {\includegraphics[width= 0.3\textwidth]{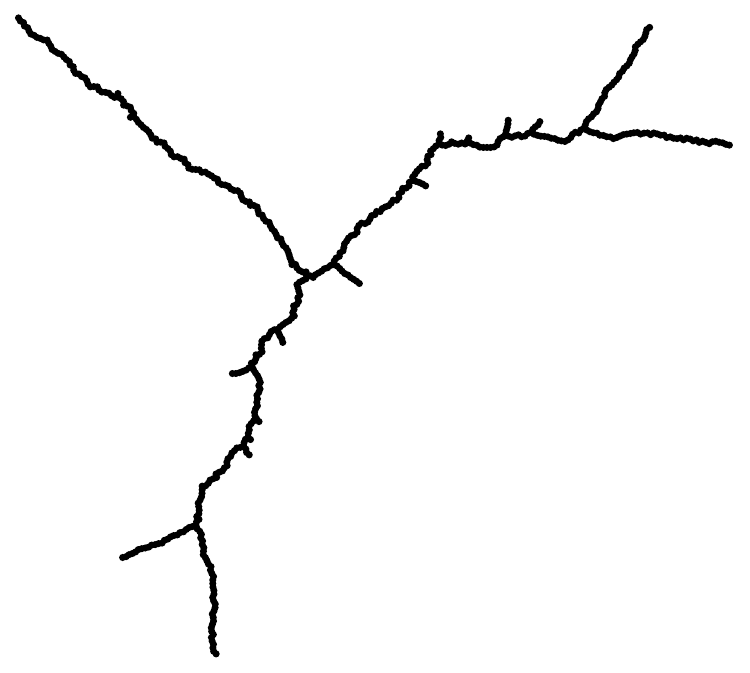}};
\node at (9,0) {\includegraphics[width= 0.3\textwidth]{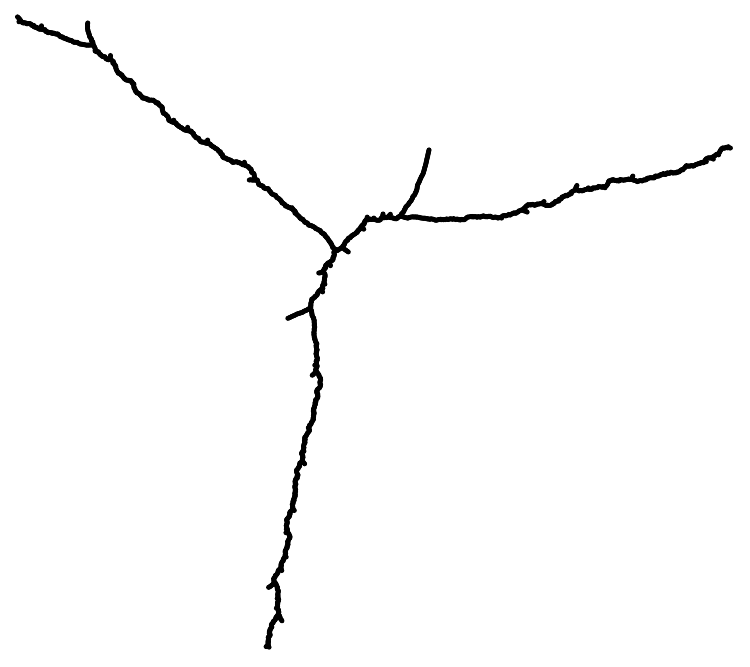}};
\draw [] (-1.3,-1.3) -- (1.3,-1.3) -- (1.3,1.3) -- (-1.3,1.3) -- (-1.3,-1.3);
\draw [] (3.4, -0.1) -- (4.2, -0.1) -- (4.2, 0.7) -- (3.4, 0.7) -- (3.4, -0.1);
\draw (2.1, -1.7) -- (5.9, -1.7) -- (5.9, 1.7) -- (2.1, 1.7) -- (2.1, -1.7);
\draw (1.3, -1.3) -- (3.4, -0.1);
\draw (1.3, 1.3) -- (3.4, 0.7);
\draw (5.9, -1.7) -- (8.45, -0.05);
\draw (5.9, 1.7) -- (8.45, 0.75);
\draw (8.45, -0.05) -- (9.25, -0.05) -- (9.25, 0.75) -- (8.45 , 0.75) -- (8.45, -0.05);
\end{tikzpicture}
\caption{$n=100$ (left), $n=500$ (middle) and $n=2000$ (right) steps of the evolution of GFA with $\alpha = \infty$ tree.}
\label{fig:5}
\end{figure}
\end{center}

\begin{proposition}
$x_{n+1}$ can only attach to particles in the convex hull of $\left\{x_1, \dots, x_{n}\right\}$. If $x_k$ is a particle in the complex hull with (interior) opening angle $\alpha$ then
$$ \mathbb{P}\left(x_{n+1}~\emph{attaches itself to}~x_k \right) = \frac{\pi - \alpha}{2\pi}.$$
\end{proposition}

\begin{center}
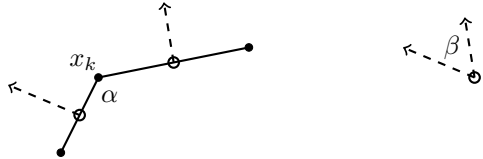
\begin{figure}[h!]
\begin{tikzpicture}
\filldraw (0,0) circle (0.05cm);
\filldraw (2,0.4) circle (0.05cm);
\filldraw (-0.5,-1) circle (0.05cm);
\draw [thick] (-0.5, -1) -- (0,0) -- (2, 0.4);
\node at (0.15, -0.25) {$\alpha$};
\node at (-0.2, 0.2) {$x_k$};
\draw[thick]  (1, 0.2) circle (0.07cm);
\draw[thick]  (-0.25, -0.5) circle (0.07cm);
\draw[thick, dashed, ->] (1, 0.2) -- (0.87, 1);
\draw[thick, dashed, ->] (-0.25, -0.5) -- (-1.2, -0.1);
\draw[thick]  (5, 0) circle (0.07cm);
\draw[thick, dashed, ->] (5, 0) -- (4.87, 0.8);
\draw[thick, dashed, ->] (5, 0) -- (4.05, 0.4);
\node at (4.7, 0.4) {$\beta$};
\end{tikzpicture}
\caption{Left: the likelihood of attaching itself to an existing particle $x_k$ on the convex hull is $\beta/(2\pi) = (\pi - \alpha)/(2\pi)$}
\label{fig:12}
\end{figure}
\end{center}

We note that this also means that the simulation of GFA for $\alpha = \infty$ is numerically simple: it suffices to keep track of the convex hull of the existing set of points, add $x_{n+1}$ and then recompute the new convex hull. Since the cardinality of the convex hull seems to be, asymptotically, bounded, computation of $x_{n+1}$ can be as cheap as $\mathcal{O}(1)$. This allows for a construction of trees with a large number of points.

\begin{center}
\begin{figure}[h!]
\begin{tikzpicture}
\node at (0,0) {\includegraphics[width= 0.21\textwidth]{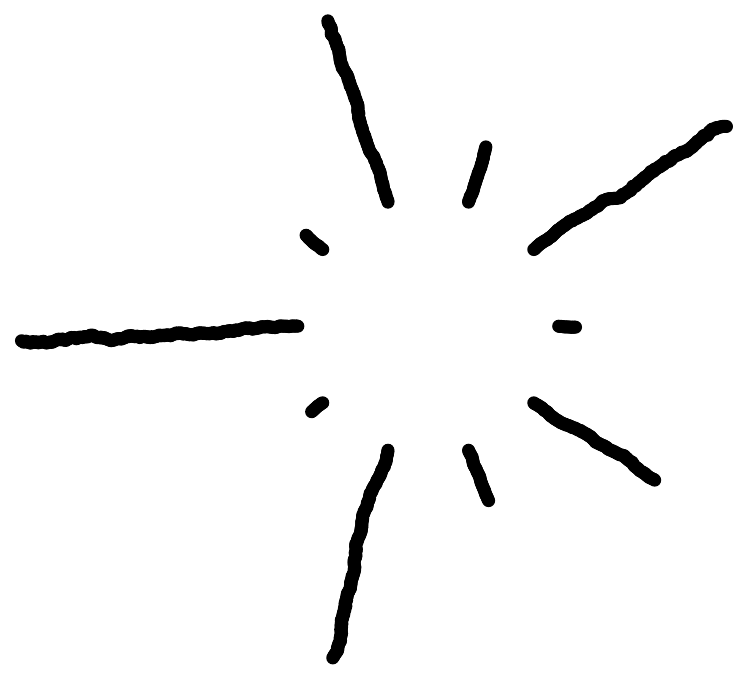}};
\node at (3.2,0) {\includegraphics[width= 0.27\textwidth]{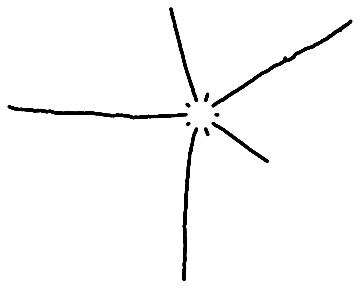}};
\node at (6.4,0) {\includegraphics[width= 0.25\textwidth]{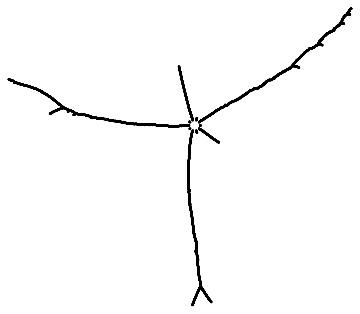}};
\node at (9.6,0) {\includegraphics[width= 0.26\textwidth]{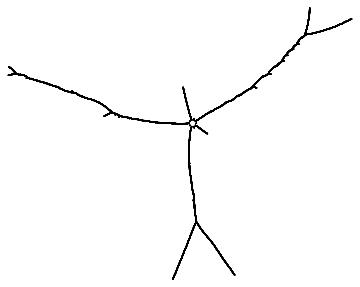}};
\end{tikzpicture}
\caption{Starting with the 10 vertices of a regular polygon at distance 100 from the origin. The evolution after 1.000, 5.000, 15.000 and 25.000 steps, respectively.}
\label{fig:6}
\end{figure}
\end{center}

This appears to be an interesting dynamical system. The growth of its diameter is approximately inversely proportional to the number of points that constitute its convex hull. It is intuitively clear that the dynamical system does not favor a large number of points in the convex hull which then suggests ballistic growth. The reason for the small number of points in the convex hull is, heuristically, the following: one of the existing points in the convex hull is bound to have a smallest angle and will then attract the largest growth. If all opening angles are the same, then random fluctuations will naturally lead to some of them being a little bit larger than others and then the natural drift sets in. In examples, see Fig. \ref{fig:6}, one tends to observe a natural tendency to shapes that appear roughly like three or four long branches that meet roughly at an equal angle. However, this does not seem to be a stable configuration and one sometimes observes shorter segments trying to branch off (and sometimes they succeed). It could be an interesting avenue for further research to understand the case $\alpha = \infty$ better.

\subsection{What to expect in higher dimensions.} Some of our arguments employ the particularly simple geometry/topology in two dimensions but it seems conceivable that many of the arguments extend to higher dimensions. We note that there is one particularly simple case, the case of
$$ E(x) = \sum_{i=1}^{n} \frac{1}{\|x-x_i\|^{d-2}} \qquad \mbox{in}~\mathbb{R}^d~\mbox{with}~d \geq 3.$$
The main reason is, unsurprisingly, that $E$ is now a harmonic function and $\Delta E = 0$. This allows us to use Lemma \ref{lem:main} almost verbatim while avoiding the use of Lemma \ref{lem:ang} (whose generalization to higher dimensions seems nontrivial).

\begin{theorem}
Consider GFA in $\mathbb{R}^d$, $d \geq 3$, and $\alpha = d-2$. Given $\left\{x_1, \dots, x_n\right\} \subset \mathbb{R}^d$,
$$ \max_{1 \leq i \leq n} \mathbb{P}\left(\emph{new particles hits}~x_i\right) \leq c_{d}  \cdot n^{\frac{1}{d} - 1}$$
and
$$ \emph{diam} \left\{x_1, \dots, x_n \right\} \leq c_{d} \cdot n^{\frac{d-1}{d}}.$$
\end{theorem}
We note that both estimates are clearly optimal: for \textit{any} aggregation scheme, no Beurling estimate can be better than $n^{\frac{1}{d} - 1}$ and this can be seen as follows: consider $n$ points arranged in some lattice structure so that only $\sim n^{(d-1)/d}$ are exposed as boundary points, by pigeonhole principle at least one of them is hit with likelihood $n^{-(d-1)/d}$. Likewise, any set of $n$ points that are $1-$separated has diameter at least $\sim n^{(d-1)/d}$. Theorem 4 suggests a phase transition when $\alpha = d-2$. One could therefore be led to expect GFA to produce ball-like shapes when $\alpha \leq d-2$ and do something else entirely when $\alpha > d-2$.

\subsection{Problems.} \label{sec:que} 
These results suggest a large number of questions. We quickly list some and note that there are many others.

\begin{enumerate}
\item We know the growth rate for the logarithmic case $\alpha =0$ which grows at the smallest possible rate. Is the limiting shape is a disk?
\item What is the sharpest possible Beurling estimate for GFA with $0 \leq \alpha < \infty$? Is the extremal set a set of $n$ points on a line?
\item How does the growth rate of GFA depend on $\alpha$? This is likely very difficult and already heuristic/numerical guesses might be of interest. We know that, for $\alpha$ close to 0, the rate is $\leq n^{\beta}$ with $1/2 \leq \beta \leq 1/2 + \alpha$. 
\item Related to the previous question, are there fast numerical methods to simulate GFA? The fast multipole method appears to be natural candidate.
\item Is the growth rate sub-ballistic for all $1 \leq \alpha < \infty$? Or is there a threshold parameter $0 < \alpha_0 < \infty$ such that GFA becomes ballistic for $\alpha \geq \alpha_0$?
\item Is it possible to prove a ballistic growth rate when $\alpha = \infty$? Numerically, this seems to be obviously the case. Do initial conditions matter?
More generally, what can be proven about GFA when $\alpha = \infty$?
\item There is a notion that DBM and Hastings-Levitov might belong to the same universality class. Does GFA fit into the picture or is it something else entirely? Is there a way in which GFA with parameter $\alpha$ behaves like DBM with parameter $\eta = \eta(\alpha)$? DBM seems to undergo a phase transition \cite{phase} for $\eta = 4$, does a similar phase transition happen for GFA?
\item Natural aspects of interest when studying GFA might be the evolution of $E(x)$, $\nabla E(x)$ and the distribution of the critical points. Is there anything rigorous that can be said about these objects?
\item  What happens if the potential $\phi(\|x-x_i\|)= \|x-x_i\|^{-\alpha}$ is replaced with some other monotonically decreasing and convex radial function? As long as there is control on the sign of the Laplacian of the radial function, some of our arguments might generalize. Is there a universality phenomenon in the sense that the evolution depends only weakly on $\phi$?
\item One could naturally consider all these questions in dimension $d \geq 3$ and many of our arguments will generalize (see also \S 1.6). As a tendency, things do not become easier in higher dimensions. However, it appears as if the simulation of the $\alpha = \infty$ tree remains equally simple in higher dimensions and it would be of interest to see what can be said.
\end{enumerate}

\begin{center}
\begin{figure}[h!]
\begin{tikzpicture}
\node at (0,0) {\includegraphics[width= 0.25\textwidth]{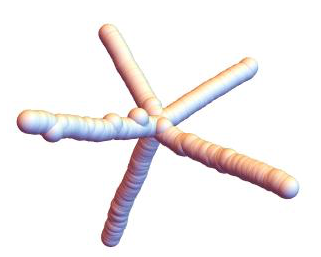}};
\node at (3,0) {\includegraphics[width= 0.25\textwidth]{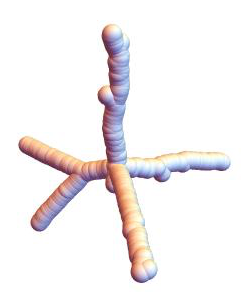}};
\node at (6,0) {\includegraphics[width= 0.25\textwidth]{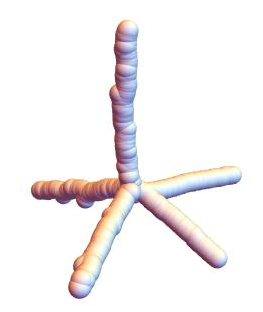}};
\node at (9,0) {\includegraphics[width= 0.25\textwidth]{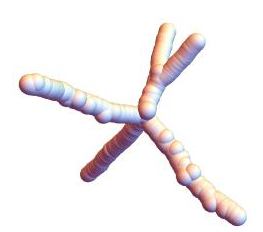}};
\end{tikzpicture}
\caption{Four examples of GFA with $\alpha=\infty$ in $\mathbb{R}^3$: the first 10.000 points. We see a tendency towards 5 tentacles.}
\label{fig:tree}
\end{figure}
\end{center}

\vspace{-30pt}

\section{Proof of Theorem 1}
This section is structured as follows. \S 2.1 details what is known about the set of critical points of $E(x)$ and discusses the role of property (P). \S 2.2 presents a quick argument that a finite number of critical points cannot have a big impact on the gradient flows (a more general version of the statement follows as a byproduct from our main argument and is later given as Lemma \ref{lem:dist} in \S \ref{sec:byproduct}). \S 2.3 shows that the set of gradient flows that end up touching a particular particle have to be topologically simple.
\S 2.4 gives a proof of Theorem 1.
\subsection{The critical set.}
We quickly discuss the set of critical points of
$$ E(x) =  \sum_{i=1}^{n}  \frac{1}{\|x-x_i\|^{\alpha}}$$
which is the set of points $x$ for which $\nabla E(x) = 0$. It can also be written as
$$ \mbox{set of critical points}  = \left\{x \in \mathbb{R}^2: \sum_{i=1}^{n} \frac{x-x_i}{\|x-x_i\|^{\alpha+2}} = 0\right\}.$$
The case $\alpha = 0$, corresponding to logarithmic energy, is completely understood: when $\alpha = 0$, then critical points can be written as the derivative of a complex-valued polynomial and there are at most $n-1$ distinct critical points because a complex polynomial of degree $n-1$ has exactly $n-1$ roots (this observation is sometimes attributed to Gauss). Note that this argument is restricted to $\mathbb{R}^2$.

Let us now consider the case $\alpha > 1$. As was only recently re-discovered \cite{gabrielov}, a particular special case dates back to James Clerk Maxwell \cite[Section 113]{maxwell} who claimed that for $\left\{x_1, \dots, x_n\right\} \subset \mathbb{R}^3$, 
$$\# \mbox{critical points of} \quad E(x) = \sum_{i=1}^{n} \frac{1}{\|x-x_i\|}  \leq (n-1)^2.$$
J. J. Thomson, when proofreading the book in 1891, added the remark `I have not been able to find any place where this result is proved.'  The problem is already interesting when $n=3$: the problem is then really two-dimensional since critical points can only occur in the convex hull of the three points. The question is thus whether any set of $n=3$ points in $\mathbb{R}^2$ leads to $E(x)$ with $\alpha = 1$ having at most 4 critical points. This is true and a rather nontrivial 2015 result of Tsai \cite{tsai}. For general $n$, it is not even known whether the critical set is always finite, not even in the plane (a situation accurately described by Shapiro \cite{shap} as `very irritating'). Under a non-degeneracy assumption, Killian \cite{killian} (improving work of Gabrielov, Novikov, Shapiro \cite{gabrielov}) showed that $n$ points in $\mathbb{R}^2$ generate at most $2^{2n-2} (3n-2)^2$ critical points. We refer to Gabrielov, Novikov, Shapiro \cite{gabrielov}, a recent result of Zolotov \cite{zolotov} and references therein for more details.\\

 We note that our setting suggests further relaxations of property (P). For example, suppose the set $\left\{x_1, \dots, x_n\right\}$ fails to have property (P). Then it appears to be quite conceivable that the set of points $x$ with the property that $\left\{x_1, \dots, x_{n-1}, x\right\}$ fails to have property (P) is necessarily small (say, measure 0). Any failure of property (P) should be very delicate. A result like this might not be out of reach and would suffice to show that property (P) is automatically satisfied along GFA with probability 1.  It is also conceivable that a more measure-geometric analysis of the set of critical points may lead to another relaxation. We have not pursued these possibilities at this time, they may be interesting future directions.

\subsection{Gradient flows and critical points.} \label{sec:grad}
We quickly argue that a fixed critical point does not capture a positive mass of gradient flows.  We give two different arguments, one based on stable/unstable manifolds that works for $\alpha > 0$ and a second argument in the case of $\alpha = 0$. The second argument is more robust and can also be made to work in the case $\alpha > 0$ where it can be used to produce a result of independent interest which we postpone to Lemma \ref{lem:dist} in \S \ref{sec:byproduct} (since it relies on some other ideas that have yet to be introduced). 
Recall that
$$ \nabla E(x) = \alpha \sum_{j=1}^{n} \frac{x_j - x}{\|x_j - x\|^{\alpha+2}} \quad \mbox{and} \quad  \Delta E(x) = \alpha^2 \sum_{j=1}^{n} \frac{1}{\|x - x_j\|^{\alpha+2}} \geq 0.$$
Assume that $\nabla E(x^*) = 0$. Since $\alpha > 0$, we have $\Delta E(x^*) > 0$ and can perform a local Taylor expansion of the energy around $x^*$
 $$ E(x^* + y) = E(x^*) + \frac{1}{2} \left\langle y, ((D^2 E)(x^*))(y) \right\rangle + \mathcal{O}(\|y\|^3).$$
Since $\mbox{tr}((D^2 E)(x^*)) = \Delta E(x^*) > 0$, know that the symmetric $2 \times 2$ matrix $(D^2 E)(x^*)$ has at least one positive eigenvalue (since the sum of the two eigenvalues is positive). If it has two positive eigenvalues, then $E$ has a strict local minimum in $x^*$. This means that the gradient flow started at $\infty$ can never reach $x^*$: here, it is important we are always considering gradient \textit{ascent} which necessarily avoids local minima (gradient descent very well might end up in a local minimum). This means that the only relevant remaining case is when $((D^2 E)(x^*))$ has one positive and one non-positive eigenvalue. After a change of variables, we may assume without loss of generality (after a translation) that $x^* = 0$ and (after a rotation) that the eigenvectors of $((D^2 E)(x^*))$ are given by $(1,0)$ and $(0,1)$. Then
 $$ E(y) = E(0) + \frac{\lambda_1}{2} y_1^2 + \frac{\lambda_2}{2} y_2^2  + \mathcal{O}(\|y\|^3)$$
 where $\lambda_1 > 0$ and $\lambda_2 \leq 0$. Then
 $$ \nabla E(y) = ( \lambda_1 y_1,  \lambda_2 y_2) +  \mathcal{O}(\|y\|^2).$$
There is one stable and one unstable direction which means that there is exactly one gradient flow line that can end up in the critical point. In particular, if there are finitely many critical points, then this set of trajectories has measure 0.  The case $\alpha = 0$ is different (since the Laplacian is 0) and a different argument is needed; we will provide a more general statement (that also works for $\alpha \geq 0$) as a consequence of the main argument in the paper as Lemma \ref{lem:dist} in \S \ref{sec:byproduct}.


\subsection{A topological Lemma.} We establish a basic topological fact showing that the set of points at distance $\|x\| = r \gg n $ from the origin whose gradient flow transports them to a fixed existing particle cannot be too complicated.
\begin{lemma} \label{lem:topo}
Let $\left\{x_1, \dots, x_n\right\} \subset \mathbb{R}^2$ be a set of $n$ particles satisfying property $(P)$. Then, for $1 \leq i \leq n$ arbitrary and $r \geq 10n$, the set
$$ A_r = \left\{\|x\| = r: \mbox{gradient flow ends up touching}~x_i \right\}$$
is, up to a finite number of points, the union of at most 6 open intervals.
\end{lemma}

\begin{proof}
We start by considering the circle of radius $1/2$ centered at $x_i$. By construction of the set, it is touched by at least one other circle of radius $1/2$ centered
at some other point. It may be touched by more than one circle but it cannot be touched by more than 6 circles. Therefore the set
$$ \left\{\|x - x_i\| = 1/2 \right\} \setminus \bigcup_{k=1 \atop k \neq i} \left\{\|x-x_k\|=1/2\right\}$$
is almost all of $\left\{\|x - x_i\| = 1/2 \right\}$ with at most 6 points removed. In particular, it can be written as the union of at most 6 connected sets which we label $C_{i,1}, C_{i,2}, \dots$.

\begin{center}
\begin{figure}[h!]
\begin{tikzpicture}[scale=0.7]
\filldraw (0,0) circle (0.05cm);
\node at (0, -0.3) {\Large $x_i$};
\draw[thick] (0,0) circle (1cm);
\filldraw (-2,0) circle (0.05cm);
\draw[thick] (-2,0) circle (1cm);
\draw[thick] (1.41,1.41) circle (1cm);
\filldraw (1.41,1.41) circle (0.05cm);
\draw [ultra thick,domain=180:360] plot ({cos(\x)}, {sin(\x)});
\draw [ultra thick,domain=0:45] plot ({cos(\x)}, {sin(\x)});
\node at (1.3, -1.2) {$C_{i,1}$};
\filldraw (6,-1) circle (0.05cm);
\node at (6.3, -1) {\Large $x$};
\filldraw (6,1) circle (0.05cm);
\node at (6.3, 1) {\Large $y$};
\draw [dashed] (6,-1) -- (6,1);
\draw [thick] (6,1) to[out=180, in=45] (4,0.5) to[out=225, in=350] (0.95, -0.2);
\draw [thick] (6,-1) to[out=180, in=45] (4,-0.5) to[out=225, in=350] (0.7, -0.72);
\node at (5,0) {$\Omega$};
\end{tikzpicture}
\caption{Creating a trapping region $\Omega$.}
\label{fig:8}
\end{figure}
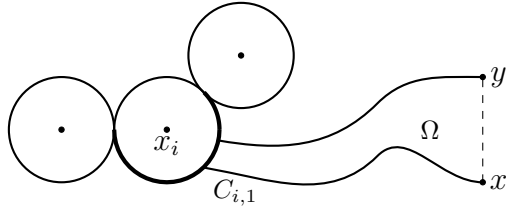
\end{center}

 Let us now assume that
$A_r$ contains at least 7 open intervals (separated by more than just a finite number of points). Then, using the pigeonhole principle, there are two points $x,y \in A_r$ that belong to two different open intervals in $A_r$ which, under the gradient flow, get sent to the same connected set $C_{i,k}$. The main idea of the argument is now shown in Fig. \ref{fig:8}. Starting the gradient flow for any point between $x$ and $y$ now leads a flow that is trapped between the two existing gradient flows in the region $\Omega$. This can be seen as follows: first, because $r \geq 10n$ and thus all the forces $\nabla E$ point inside the disk, it is clear that the gradient flow can never escape through $A_r$. It can never touch the gradient lines induced by $x$ and $y$ by uniqueness. It can also not get within distance $1$ of any of the existing points because the existing particles are connected. It may get stuck in a critical point but, by assumption, there are only finitely many of those which then shows that, in the set $A_r$, the open interval containing $x$ and the open interval containing $y$ may actually be joined into a single open interval (up to finitely many points).
\end{proof}

\subsection{Proof of Theorem 1}

\begin{proof} We will prove a slightly stronger statement: for any point $x_i$ and any interval $J \subset \left\{y: \|x_i -y\| = 1\right\}$, there exists an asymptotic likelihood of ending up in that interval (in the sense of the limit of the measure existing when $r \rightarrow \infty$). Thus, by considering random points at distance $\|x\| = r$, there is an emerging limit as $r \rightarrow \infty$ of where these points impact on the boundary. We can, without loss of generality, assume that $J$ is an interval on the disk with the property that no element in the interval is distance $1/2$ from any other point (if not, we can always split the interval into two intervals and argue for each one separately). 

\begin{center}
\begin{figure}[h!]
\begin{tikzpicture}
\filldraw (0,0) circle (0.05cm);
\node at (0,0) {\includegraphics[width=0.13\textwidth]{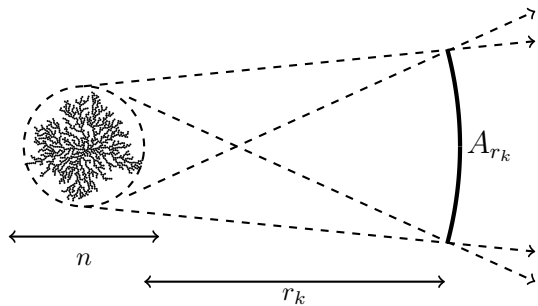}};
\draw[thick, <->] (-1,-1.2) -- (1,-1.2);
\node at (0, -1.5) {$n$};
\draw[thick, <->] (0.8,-1.8) -- (4.8,-1.8);
\node at (2.8, -2) {$r_k$};
\draw[dashed, thick] (0,0) circle (0.8cm);
\draw [ultra thick,domain=345:359.99] plot ({5*cos(\x)}, {5*sin(\x)});
 \draw [ultra thick,domain=0:15] plot ({5*cos(\x)}, {5*sin(\x)});
 \node at (5.4, 0) {\Large $A_{r_k}$};
 \draw[dashed, thick, ->] (0.6, -0.65) -- (6, 1.8);
  \draw[dashed, thick, ->] (0.6, 0.65) -- (6, -1.8);
   \draw[dashed, thick, ->] (0, 0.8) -- (6, 1.4);
      \draw[dashed, thick, ->] (0, -0.8) -- (6, -1.4);
\end{tikzpicture}
\caption{Sketch of the proof of Theorem 1.}
\label{fig:9}
\end{figure}
\end{center}
Repeating the argument from Lemma \ref{lem:topo} we see that the set $A_r$, which is now the set of points on $\left\{x: \|x\| = r\right\}$ that end up flowing into $J$, are a connected set (up to at most finitely many points that are transported to critical points, here we use property $(P)$). There are two cases:
$$ \mbox{either} \quad \lim_{r \rightarrow \infty} \frac{|A_r|}{2\pi r} = 0 \qquad \mbox{or not}.$$
If the limit is 0, we have the desired convergence (the likelihood of hitting the interval $J$ is then 0). Suppose now that the limit is not 0. Then there exists $\varepsilon_0$ such that for a sequence $(r_k)_{k=1}^{\infty}$ tending to infinity $|A_{r_k}| \geq \varepsilon_0 r_k$. We pick such a very large radius $r_k$, say one that is so large that it ensures that $|A_{r_k}| \geq e^n$, and then argue as illustrated in Figure \ref{fig:9}. We don't exactly know how the existing $n$ particles are distributed inside the original disk of radius $\leq n$ but they all exert a radial force. 
This shows that, at the endpoints of $A_{r_k}$, we can actually predict up to an angle of size $\sim n/r_k$, how the flow lines will look like (by assuming, simultaneously, that all the existing points are concentrated in the northern and the southern tip of the disk). This means that the asymptotic probability is actually determined up to a factor of $\sim n/r_k$. This can be made arbitrarily small by letting $r \rightarrow \infty$ which forces convergence of $|A_r|/r$.
\end{proof}

\section{Proof of Theorem 2}
The main ingredient is Lemma \ref{lem:main} in \S 3.1.  \S 3.2 proves a geometric result, Lemma \ref{lem:ang}, which then, together with Lemma \ref{lem:main} proves Theorem 2 in \S 3.3. A different look at Lemma \ref{lem:main} leads to Lemma \ref{lem:dist} in \S 3.4.

\subsection{The Main Step of the Argument} 
\begin{lemma} \label{lem:main} Let $x_i$ be arbitrary and define, for $r \geq 10n$,
$$ A_r = \left\{x \in \mathbb{R}^2: \|x\| = r ~ \mbox{and gradient flow leads to}~x_i \right\}.$$
Then, for $c_{\alpha}>0$ depending only on $\alpha > 0$,
$$ \frac{|A_{r}|}{2 \pi r} \leq c_{\alpha} r^{\alpha}  n^{-\frac{\alpha}{2} - \frac{1}{2}}.$$
\end{lemma}
\begin{proof} 
We define the set $\Omega_r \subset \mathbb{R}^2$ as all the points in $\mathbb{R}^2$ that are element of some gradient flow started by some point in $A_i$ (and ending up within distance $1$ from $x_i$). Green's Theorem implies
$$ \int_{\Omega_r} \Delta E ~dx = \int_{\partial \Omega_r} \nabla E \cdot dn,$$
where $n$ is the normal vector pointing outside. 

\begin{center}
\begin{figure}[h!]
\begin{tikzpicture}
\filldraw (0,0) circle (0.05cm);
\node at (0, -0.3) {\Large $x_i$};
\draw[dashed, thick] (0,0) circle (1cm);
\draw [ultra thick,domain=350:359.99] plot ({5*cos(\x)}, {5*sin(\x)});
 \draw [ultra thick,domain=0:10] plot ({5*cos(\x)}, {5*sin(\x)});
 \draw [ultra thick,domain=350:359.99] plot ({cos(\x)}, {sin(\x)});
 \draw [ultra thick,domain=0:10] plot ({cos(\x)}, {sin(\x)});
 \node at (5.4, 0) {\Large $A_r$};
 \draw [thick] (1, 0.2) -- (4.95, 0.85);
  \draw [thick] (1, -0.2) -- (4.95, -0.85);
 \node at (3, 0) {\Large $\Omega_r$};
 \draw[thick, <->] (-1,-1.2) -- (1,-1.2);
\node at (0, -1.5) {$1$};
 \draw[thick, <->] (-1,-1.7) -- (2,-1.7);
\node at (0.5, -2) {$n$};
\draw[thick, <->] (2.5,-1.5) -- (4.8,-1.5);
\node at (3.8, -1.8) {$r \geq 10n$};
\end{tikzpicture}
\caption{Sketch of the proof.}
\label{fig:10}
\end{figure}
\end{center}

The gradient and the Laplacian of the energy are given by
$$ \nabla E(x) = \alpha \sum_{j=1}^{n} \frac{x_j - x}{\|x_j - x\|^{\alpha+2}} \quad \mbox{and} \quad  \Delta E(x) = \alpha^2 \sum_{j=1}^{n} \frac{1}{\|x - x_j\|^{\alpha+2}} \geq 0.$$
This implies
$$ \int_{\partial \Omega_r} \nabla E \cdot dn \geq 0.$$
The remainder of the argument is dedicated to deducing information from this inequality.
We decompose the boundary of $\Omega_r$ into three sets
$$ \partial \Omega_r = (\partial \Omega_r \cap \left\{z: \|z-x_i\| = 1\right\}) \cup A_r\cup  \mbox{remainder},$$
where $\partial \Omega_r \cap \left\{z: \|z-x_i\| = 1\right\}$ is the part of the boundary of $\Omega_r$ that is within distance 1 of $x_i$, the set $A_r$ is the initial starting points of the gradient flow that are far away and the remainder is everything else (comprised of a union of gradient flows). We note that
$$  \int_{\mbox{\tiny remainder}} \nabla E \cdot dn = 0$$
because the remainder is given by gradient flows flowing in direction $\nabla E$: the normal derivative vanishes everywhere.
We arrive at
$$  \int_{\partial \Omega_r \cap \left\{z: \|z-x_i\| = 1\right\}} \nabla E \cdot dn \geq -   \int_{A_r} \nabla E \cdot dn.$$
We will now bound first the left-hand side from above and then the right-hand side from below.
On the part of the boundary that is distance 1 to $x_i$, the normal vector can be written as
$(x_i - x)$ and thus
$$  \int_{\partial \Omega_r \cap \left\{z: \|z-x_i\| = 1\right\}} \nabla E \cdot dn = \int_{\partial \Omega_r \cap \left\{z: \|z-x_i\| = 1\right\}} \left\langle  \alpha \sum_{j=1}^{n} \frac{x_j - x}{\|x_j - x\|^{\alpha+2}}, x_i  - x \right\rangle d\sigma,$$
where $\sigma$ is the arclength measure. We conclude
$$  \left\langle \sum_{j=1}^{n} \frac{x_j - x}{\|x_j - x\|^{\alpha+2}}, x_i  - x \right\rangle =   \sum_{j=1}^{n} \frac{  \left\langle   x_j - x, x_i - x \right\rangle}{\|x_j - x\|^{\alpha+2}} \leq   \sum_{j=1}^{n} \frac{1}{\|x_j - x\|^{\alpha+1}},$$
where the last step used Cauchy-Schwarz and the fact that $\|x_i - x\|=1$. Thus
\begin{align*}
 \left|   \int_{\partial \Omega_r \cap \left\{z: \|z-x_i\| = 1\right\}} \nabla E \cdot dn  \right|  &\leq  \alpha \int_{\partial \Omega_r \cap \left\{z: \|z-x_i\| = 1\right\}}    \sum_{j=1}^{n} \frac{1}{\|x_j - x\|^{\alpha+1}} d\sigma\\
 &\leq 2 \pi \alpha \max_{\partial \Omega_r \cap \left\{z: \|z-x_i\| = 1\right\}} \sum_{j=1}^{n} \frac{1}{\|x_j - x\|^{\alpha+1}}.
 \end{align*}
 At this point, we use a packing argument. The existing points $\left\{x_1, \dots, x_n\right\}$ are $1-$separated, meaning $\|x_i - x_j\| \geq 1$ when $i \neq j$. Likewise, the new point $x$ is also distance $\geq 1$ from all the existing points. The sum is therefore maximized if all the existing points $x_i$ are packed as tightly as possible around $x$. Switching to polar coordinates around $x$, we have that for some absolute constants,
$$ \max_{\partial \Omega_r \cap \left\{z: \|z-x_i\| = 1\right\}}  \sum_{j=1}^{n} \frac{1}{\|x_j - x\|^{\alpha+1}} \leq c^*_{\alpha} \sum_{\ell=1}^{\sqrt{n}} \ell \frac{1}{\ell^{\alpha+1}} \leq c_{\alpha} \cdot n^{\frac{1}{2}-\frac{\alpha}{2}}.$$
One could make the constant $c_{\alpha}$ explicit, indeed, since the sphere-packing problem is solved in two dimensions, it stands to reason that the optimal constant should be attained for the hexagonal lattice (a result in a similar spirit has been shown in \cite{grabner}); we have not tracked explicit constants anywhere and will not start here, however, this may be a useful avenue to pursue in regards to whether the asymptotic shape is a disk when $\alpha = 0$.
It remains to bound the second integral from below. It can be written as
$$-   \int_{A_r} \nabla E \cdot dn =  \alpha \int_{A_r} \left\langle   \sum_{j=1}^{n} \frac{x-x_j}{\|x_j - x\|^{\alpha+2}}, \frac{ x}{\|x\|} \right\rangle d\sigma.$$
Since $\|x\| \geq 10n$ and $\|x_i\| \leq n$ and thus $\|x-x_j\| \geq 9n$, we have
$$ \left\langle x-x_j, x\right\rangle = \|x-x_j\|^2 + \left\langle x-x_j, x_j\right\rangle \geq \frac{1}{2} \|x-x_j\|^2$$
and thus
$$ -   \int_{A_r} \nabla E \cdot dn  \geq \frac{\alpha}{100} \int_{A_r} \sum_{j=1}^{n} \frac{1}{\|x_j - x\|^{\alpha+1}} d\sigma.$$
Using again $r \geq 10n$ together with $\|x- x_i\| \geq r/2$, we have
$$  \frac{\alpha}{100} \int_{A_r} \sum_{j=1}^{n} \frac{1}{\|x_j - x\|^{\alpha+1}} d\sigma \geq c_{\alpha}^* \frac{n}{r^{\alpha+1}}|A_r|.$$
Combining the upper bound and the lower bound, we obtain
$$ \frac{|A_{r}|}{2 \pi r} \leq c_{\alpha} r^{\alpha}   n^{-\frac{\alpha}{2} - \frac{1}{2}}.$$
\end{proof}

\subsection{Bounding the likelihood at infinity.}
In order to conclude the result, we need to show that the size of $A_{r}$ for finite $r$ is related to the probability of selecting the particle (which is governed by the asymptotic behavior of $|A_r|/r$ as $r \rightarrow \infty$). 

\begin{lemma} \label{lem:ang}
Suppose $\left\{x_1, \dots, x_n\right\} \subset \mathbb{R}^2$ satisfies property (P). Then the asymptotic likelihood of hitting particle $x_i$ satisfies
$$  p = \lim_{r \rightarrow \infty} \frac{|A_r|}{2 \pi r} \leq  100 \cdot \frac{\left| A_{ \frac{100n}{p}} \right|}{ \frac{100n}{p}}.$$
\end{lemma}

\begin{proof} 
We may think of each individual existing particle as exerting a force and that these forces, once
one is far away from the set, are contained in a convex cone. We refer to Fig. \ref{fig:11} for a sketch
of the argument. It remains to compute the opening angle of the cone.

\begin{center}
\begin{figure}[h!]
\begin{tikzpicture}
\filldraw (0,0) circle (0.05cm);
\node at (0,0) {\includegraphics[width=0.13\textwidth]{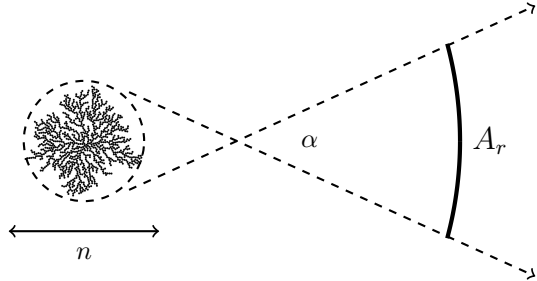}};
\draw[thick, <->] (-1,-1.2) -- (1,-1.2);
\node at (0, -1.5) {$n$};
\draw[dashed, thick] (0,0) circle (0.8cm);
\draw [ultra thick,domain=345:359.99] plot ({5*cos(\x)}, {5*sin(\x)});
 \draw [ultra thick,domain=0:15] plot ({5*cos(\x)}, {5*sin(\x)});
 \node at (5.4, 0) {\Large $A_r$};
 \draw[dashed, thick, ->] (0.6, -0.65) -- (6, 1.8);
  \draw[dashed, thick, ->] (0.6, 0.65) -- (6, -1.8);
  \node at (3, 0) {$\alpha$};
\end{tikzpicture}
\caption{Sketch of the proof of Lemma \ref{lem:ang}. Connected components flowing into $A_r$ have to be contained inside the cone.}
\label{fig:11}
\end{figure}
\end{center}

We assume, for convenience, the disk of radius $n$ to be centered in the origin and, again for convenience, that one (of possibly several but at most 6, see Lemma \ref{lem:topo}) intervals is centered symmetrically
around the $x-$axis. We choose $r$ to be large. The length of the set $A_r$ is then, for $r$ very large, asymptotically given by $2 \pi p r$, where $0 \leq p \leq 1$ is the asymptotic probability. We choose $r$ so large that $2 \pi p r \gg e^n \geq n$.
 The two lines in Fig. \ref{fig:11} can be approximated by $y= px - n$ and $y= -px + n$. These lines meet at $x = n/p$. Plugging in $x = 100n/p$ shows that
$$ |A_{100 \frac{n}{p}}| \geq 10 n = 100 \frac{n}{p} \frac{p}{10} \qquad \mbox{and therefore} \qquad \frac{ |A_{100 \frac{n}{p}}|}{100 \frac{n}{p}} \geq \frac{p}{10}.$$
 
\end{proof}

\subsection{Proof of Theorem 2}
\begin{proof} With all these ingredients in place, the proof of Theorem 2 is now straight-forward.
Using Lemma \ref{lem:main} and Lemma \ref{lem:ang}, we have
$$  p = \lim_{r \rightarrow \infty} \frac{|A_r|}{2 \pi r} \leq  100 \cdot \frac{\left| A_{ \frac{100n}{p}} \right|}{ \frac{100n}{p}} \leq c_{\alpha} \left( \frac{n}{p} \right)^{\alpha} n^{-\frac{\alpha}{2} - \frac{1}{2}}$$
from which we deduce $ p^{1+\alpha} \leq c\cdot  n^{\frac{\alpha-1}{2}}$ and thus $p \leq c  \cdot n^{\frac{\alpha-1}{2\alpha+2}}$.
\end{proof}
It seems likely that there is at least some room for improvement here. Lemma \ref{lem:ang} uses only basic convexity properties of the forces that are being exerted by the existing particles. It is not inconceivable that something like
$$(\diamond) \qquad \qquad  c_{\alpha, 1}\frac{|A_{10n}|}{10n} \leq \lim_{r \rightarrow \infty} \frac{|A_r|}{r} \leq c_{\alpha, 2}  \frac{|A_{10n}|}{10n} \qquad \mbox{might be true.}$$
One motivation is that if $\mu$ is a probability measure in the unit disk $D(0,1)$, then
$$ \nabla_x \int_{D(0,1)} \frac{d\mu(y)}{\|x-y\|^{\alpha}} dy \quad \mbox{for}~\|x\| \geq 10$$
 should have a greater degree of regularity than is exploited by the simple geometric argument in Lemma \ref{lem:ang}. It appears that multipole expansions might be a natural avenue to pursue in this direction. If $(\diamond)$ were the case, we would obtain the Beurling estimate $\mathbb{P} \leq  c_{\alpha} n^{\alpha/2 - 1/2}$ which is a slight improvement over Theorem 2. However, there is still no reason to assume that this would be optimal (except when $\alpha = 0$ where Theorem 2 is already optimal). Theorem 3 would then lead to a slight improvement on the growth rate and show that $\mbox{diam} \leq c_{\alpha} n^{\alpha/2 + 1/2}$. There is also no reason to assume that this improved rate would be optimal. None of this works for $\alpha \geq 1$, it seems that truly new ideas are required to deal with this regime.

\subsection{A useful byproduct} \label{sec:byproduct}
We use this section to quickly point out an interesting byproduct of the previous arguments (which can also be used to give an alternative approach to the argument from Section \S \ref{sec:grad}). 
\begin{lemma} \label{lem:dist}
Let $\alpha \geq 0$ and suppose $\left\{x_1, \dots, x_n\right\} \subset \mathbb{R}^2$ satisfies property (P). Let $y$ be a point that is at least distance 1 from all the existing points, $\|y - x_i\| \geq 1$. Then, for $\varepsilon \leq 1/10$, the likelihood of hitting the disk $B_{\varepsilon}(y)$ satisfies
$$ \mathbb{P}\left( \emph{new incoming particle hits}~B_{\varepsilon}(y) \right) \leq  c_{\alpha} \cdot \varepsilon^{\frac{1}{1 + \alpha}}  \cdot n^{\frac{\alpha-1}{2\alpha+2}}.$$
\end{lemma}
\begin{proof}
Lemma \ref{lem:topo} implies that the set $A_r \subset \left\{x: \|x\| = r \right\}$ with the property that the gradient flow ends up hitting $B_{\varepsilon}(y)$ before getting within distance 1 of any the $x_i$ must be connected. We use the argument from Lemma \ref{lem:main} and argue that
$$  \int_{\partial \Omega_r \cap \left\{z: \|z-y\| = \varepsilon \right\}} \nabla E \cdot dn \geq -   \int_{A_r} \nabla E \cdot dn.$$
Arguing as in the proof of Lemma  \ref{lem:main}, we have 
$$ \left|   \int_{\partial \Omega_r \cap \left\{z: \|z-y\| = \varepsilon \right\}} \nabla E \cdot dn \right| \leq c_{\alpha} \cdot \varepsilon \cdot n^{\frac{1}{2} - \frac{\alpha}{2}}.$$
As well as
$$ -   \int_{A_r} \nabla E \cdot dn \geq c_{\alpha} n \frac{|A_r|}{r^{\alpha+1}}.$$
Thus, as in the proof of Lemma \ref{lem:main}, we have
$$ \frac{|A_r|}{2 \pi r} \leq c_{\alpha} \varepsilon r^{\alpha} n^{-\frac{\alpha}{2} - \frac{1}{2}}.$$
Applying once more Lemma \ref{lem:ang}
$$  p = \lim_{r \rightarrow \infty} \frac{|A_r|}{2 \pi r} \leq  10 \cdot \frac{\left| A_{ \frac{100n}{p}} \right|}{ \frac{100n}{p}} \leq c_{\alpha} \varepsilon \left( \frac{n}{p} \right)^{\alpha} n^{-\frac{\alpha}{2} - \frac{1}{2}}$$
and thus
$$  p \leq c_{\alpha} \cdot \varepsilon^{\frac{1}{1 + \alpha}}  \cdot n^{\frac{\alpha-1}{2\alpha+2}}.$$
\end{proof}

We may think of Lemma \ref{lem:dist} as a measure of the maximum amount of distortion that the gradient flows can experience. In particular, letting $\varepsilon \rightarrow 0$ gives us quantitative control on the statement that critical points are hit with probability 0. The case $\alpha = 0$ is also instructive because only very little distortion can take place.

\section{Proof of Theorem 3}
The purpose of this section is to deduce the growth estimate from a Beurling estimate. We first present a simple argument showing that we can consider the problem in a suitable discrete setting on graphs. After that, we present an existing argument by Benjamini-Yadin \cite{ben} (suitably simplified and adapted to our setting). 

\subsection{Going discrete.} We first note a basic fact. GFA leads, with probability 1, to particles whose disks touch and which form a tree structure. This from the fact that a new incoming particle will only attach itself to one existing particle with likelihood 1. There are points where a new circle would touch two circles but this is a finite set of points and gradient descent shows that they lead to a set of measure 0. Moreover, since every disk in the Euclidean plane can only touch at most 6 other circles, we can furthermore deduce that the every vertex in the degree has at most 6 neighbors. We deduce the following fact.

\begin{proposition}
The adjacency structure of GFA can be embedded into an infinite 6-regular tree (taken to be rooted in the first particle $x_1$) with likelihood 1.
\end{proposition}

One way of studying the diameter of GFA is to study the diameter of the evolving induced subgraph on the infinite 6-regular tree. Note that the likelihoods of attaching a new particle/vertex to an existing vertex in the tree is still induced by the underlying continuous setting in $\mathbb{R}^2$ (one may think of GFA as running in the background and the evolution on tree as being a cartoon picture). The argument then simply uses that the diameter of the induced subgraph on the infinite tree immediately bounds the diameter of GFA by triangle inequality.

\subsection{A Concentration Inequality}
This section is almost verbatim from Benjamini-Yadin \cite{ben} and included for the convenience of the reader.
\begin{lemma} \label{lem:prob} Let $B = \sum_{k=1}^{n} Z_k$ be a sum of independent Bernoulli random variables taking values in $\left\{0,1\right\}$. Then, for any $C > 1$, we have
$$ \mathbb{P} \left( B \geq C \cdot \mathbb{E} B\right) \leq \exp\left[ - C \left( \mathbb{E}B\right) \log\left(\frac{C}{e}\right)\right].$$ 
\end{lemma}
\begin{proof}
This follows from the classical Bernstein method. For any $\alpha > 0$, we have
$$ \mathbb{E} e^{\alpha B} = \prod_{k=1}^{n} \mathbb{E} e^{\alpha Z_k} = \prod_{k=1}^{n} ((e^{\alpha} -1)\cdot  \mathbb{E} Z_k +1).$$
Now, using $\log(1+x) \leq x$, we have
\begin{align*}
 \prod_{k=1}^{n} ((e^{\alpha} -1)\cdot  \mathbb{E} Z_k +1) &= \exp\left( \sum_{k=1}^{n} \log\left( 1 + (e^{\alpha} -1)\cdot  \mathbb{E} Z_k \right) \right) \\
 &\leq  \exp\left( \sum_{k=1}^{n} (e^{\alpha} -1)\cdot  \mathbb{E} Z_k  \right) = e^{(e^{\alpha} -1) \cdot \mathbb{E} B}.
\end{align*}

Markov's inequality leads to
\begin{align*}
\mathbb{P}\left( B \geq C \cdot \mathbb{E} B \right) &= \mathbb{P}\left( e^{\alpha B} \geq e^{\alpha C \mathbb{E} B} \right) =  \mathbb{P}\left( e^{\alpha B - \alpha C \mathbb{E}B} \geq 1 \right) \\
&\leq \mathbb{E}~ e^{\alpha B - \alpha C \mathbb{E}B} = e^{- \alpha C \mathbb{E}B} \cdot \mathbb{E}~ e^{\alpha B} \leq e^{- \alpha C \mathbb{E}B}  e^{(e^{\alpha} -1) \cdot \mathbb{E} B}.
\end{align*}
It remains to optimize the function $-\alpha C + e^{\alpha} - 1$ over $\alpha$. Differentiating suggests choosing $\alpha$ so that $e^{\alpha} = C$ (which is possible since $C>1$). Plugging in, we get
$$ \mathbb{P}\left( B \geq C \cdot \mathbb{E} B \right)  \leq \exp \left[ (\mathbb{E} B) \left( - (\log{C}) C + C -1 \right) \right].$$
Since the random variables are nonnegative, we have $\mathbb{E}B \geq 0$ and can thus use 
$$- (\log{C}) C + C -1 \leq - (\log{C}) C + C = - C \log\left( \frac{C}{e} \right).$$
\end{proof}

\subsection{Growth on the tree} The argument in this subsection is also fully contained in Benjamini-Yadin \cite{ben}. We merely specialized the general argument to our setting.
We will now suppose that we have $n$ GFA particles $\left\{x_1, \dots, x_n \right\} \subset \mathbb{R}^2$. We now consider the process of adding the next $n$ particles $x_{n+1}, \dots, x_{2n}$. We first consider the subgraph $G_n \subset H_6$ of the infinite 6-regular tree and then the second induced subgraph $G_{2n} \subset H_6$. The statement we are going to prove is as follows.

\begin{lemma}
Assuming GFA satisfies a Beurling estimate
$$ \max_{1 \leq i \leq n} \mathbb{P}\left(\emph{new particles hits}~x_i\right) \leq c_{\alpha} \cdot n^{-\alpha},$$
there exists a constant $c>0$ (depending only on $c_{\alpha}$) such that
$$ \forall~v \in G_{2n} \qquad d(v, G_n) \leq c \cdot n^{1-\alpha} \qquad \mbox{with likelihood at least} \quad 1 - n \cdot e^{-c n^{1-\alpha}}.$$
\end{lemma}
\begin{proof}

The argument is based on analyzing an individual path and then taking a union bound over all paths. Let $v \in G_n$ be a vertex that is adjacent
to a vertex in $H_{6} \setminus G_n$. Consider a fixed path of length $c \cdot n^{1-\alpha}$ emanating from $v$ and never touching either itself or $G_n$. We ask ourselves: what is the likelihood that this path is being added when we add in the next $n$ points? Using $Z_i$ to denote the indicator variable of the event of next vertex in a (fixed) path being added when adding particle $v_{n+i}$, we have that
$$ Z_i \in \left\{0,1\right\} ~ \mbox{is Bernoulli with} \quad \mathbb{E} Z_i \leq c_{\alpha} n^{-\alpha} \quad \mbox{and} \quad \mathbb{E} \sum_{i=1}^{n}Z_i \leq c_{\alpha} n^{1-\alpha}.$$
The likelihood of the entire path of length $c n^{1-\alpha}$ being added is thus bounded from above by the likelihood of having $\mathbb{E} \sum_{i=1}^{n}Z_i$ exceed $ c n^{1-\alpha}$
$$ \mathbb{P}\left(\sum_{i=1}^{n} Z_i \geq c ~ n^{1-\alpha} \right) \leq \mathbb{P}\left(B \geq \frac{c ~  n^{1-\alpha}}{\mathbb{E} B} \mathbb{E} B \right)$$
Appealing to Lemma \ref{lem:prob}, we have
\begin{align*}
  \mathbb{P}\left(B \geq \frac{c ~  n^{1-\alpha}}{\mathbb{E} B} \mathbb{E} B \right) &\leq \exp\left(  - c ~ n^{1-\alpha} \log\left( \frac{c ~  n^{1-\alpha}}{\mathbb{E} B}  \right) \right) \\
  &\leq  \exp\left(  - c \cdot n^{1-\alpha} \log\left( \frac{c }{c_{\alpha}}  \right) \right).
  \end{align*}
This is the likelihood of any particular path venturing far. It remains to count the number of such paths: since we have $n$ original vertices to start from and are on a $6-$regular tree, the number of such paths is trivially bounded from above by $n 6^{c n^{1-\alpha}}$ and thus, by the union bound,
$$ \mathbb{P}\left(\mbox{long path exists}\right) \leq n \cdot e^{ c ~ n^{1-\alpha}  \left(\log{(6)}   - \log\left( \frac{c }{c_{\alpha}}\right)  \right)}.$$
Choosing $c = 1000 \cdot c_{\alpha}$, then implies the result.
\end{proof}

The growth bound now follows from the Borel-Cantelli lemma. We refer to Benjamini-Yadin \cite{ben} for details and implementations of a similar idea in more general settings.

\section{Proof of Theorem 4}
\begin{proof}
The argument is a subset of the existing arguments. We briefly summarize them in the correct order. Suppose $\left\{x_1, \dots, x_n\right\} \subset \mathbb{R}^d$ is given, $d \geq 3$ and we consider the harmonic energy
$$ E(x) = \sum_{i=1}^{n} \frac{1}{\|x- x_i\|^{d-2}}.$$
Let us pick a large radius $R \gg n \gg 1$, let $A_R$ be all points on $\left\{x: \|x\|=R\right\}$ for which the gradient flow attaches itself to particle $x_i$ and denote the union of all these flows by $\Omega$. Then
$$ \int_{\partial \Omega} \nabla E \cdot dn  = \int_{\Omega} \Delta E~ dx = 0.$$
The boundary integral can be decomposed into the two parts: the parts close to $x_i$ and the parts on the sphere $\left\{x: \|x\|=R\right\}$
$$ \partial \Omega = \partial \Omega_{x_i} \cup \partial \Omega_{R}.$$
We note that, as before, when $R \rightarrow \infty$,
\begin{align*}
\int_{\partial \Omega_R} \nabla E \cdot dn  &= (d-2+o(1)) \int_{\Omega_R} \frac{x_i - x}{\|x_i - x\|^d} dn \\
&= (1+o(1)) \frac{n}{R^{d-1}}  \int_{\Omega_R} d\sigma = (1+o(1)) \cdot n \cdot \frac{|\Omega_R|}{R^{d-1}}.
\end{align*}
Meanwhile, on the other side we argue with the triangle inequality, $1-$separation of points and a constant $c_d$ that may change its value every time it appears that
\begin{align*}
 \left| \int_{\partial \Omega_{x_i}} \nabla E \cdot dn \right| &\leq   \int_{\partial \Omega_{x_i}}  \left| \nabla E \right| \cdot d\sigma  \\
 &\leq c_d   \int_{\partial \Omega_{x_i}} \sum_{i=1}^{n} \frac{1}{\|x-x_i\|^{d-1}}  d\sigma  \leq c_d \sum_{\ell=1}^{n^{1/d}} \ell^{d-1} \frac{1}{\ell^{d-1}} \leq c_d n^{1/d}.
 \end{align*}
 This proves the Beurling estimate. As for the growth estimate, we can simply modify the existing argument as follows: a GFA tree in $\mathbb{R}^d$ can, with likelihood 1, be embedded in an infinite regular tree $H_k$ where $k$ is the kissing number of the space: that is, the maximal number a sphere of radius 1 can be simultaneously touched by other disjoint spheres with the same radius. The remainder of the argument is unchanged (though all the implicit constants do of course change).
\end{proof}

\section{Proof of the Proposition}
\begin{proof}
We first note, when picking a point $y$ at distance $R \gg n$ from the existing set of points, the set $\|y-x\| = \mbox{const}$ is a circle with curvature $1/\mbox{const}$. These circles become flat lines as $\|y\| \rightarrow \infty$ and therefore only points in the convex hull are considered at each step.
We can now consider three consecutive points on the convex hull and assume there is an (interior) angle $\alpha$ in the middle point. We can
then draw the half-lines separating closest distances between any of the points. These half-lines do not meet but, asymptotically, the only relevant questions is the opening angle $\beta$ that they create. Elementary trigonometry shows $\beta = \pi - \alpha$. The sum of the opening angles in a convex polygon with $m$ sides is
$$ \sum_{i=1}^{m} \alpha_i = (m-2) \pi.$$
Therefore, the likelihoods in our case sum up to
$$ \sum_{i=1}^{m} (\pi - \alpha_i) = m \pi - (m-2) \pi = 2 \pi.$$
Thus the likelihood of attaching itself to a new point is given by $(\pi - \alpha)/(2\pi)$.
\end{proof}

\textbf{Acknowledgment.} The author was partially supported by the NSF (DMS-2123224).

\end{document}